\documentclass[12pt]{article}

\usepackage{amsthm,tikz,amssymb, latexsym, amsmath, amscd, amsfonts, array, graphicx, lmodern, slantsc,enumerate}
\usepackage[english]{babel}        
\usepackage[all]{xy}
\CompileMatrices

\usepackage{color}

\def\TC{\protect\operatorname{TC}}
\def\bsd{\protect\operatorname{Sd}}
\def\SC{\protect\operatorname{SC}}

\newtheorem{proposition}{Proposition}[section]

\newtheorem{definition}[proposition]{Definition}
\newtheorem{theo}[proposition]{Theorem}
\newtheorem{remark}[proposition]{Remark}

\newtheorem{lema}[proposition]{Lemma}

\begin{document}

\title{Simplicial Complexity:\\piecewise linear motion planning in robotics}
\author{Jes\'us Gonz\'alez\thanks{Partially supported by Conacyt Research Grant 221221.}}
\date{\today}

\maketitle

\begin{abstract}
Using the notion of contiguity of simplicial maps, we adapt Farber's topological complexity to the realm of simplicial complexes. We show that, for a finite simplicial complex $K$, our discretized concept recovers the topological complexity of the realization $\|K\|$. Our approach is well suited for designing and implementing algorithms that search for optimal motion planners for autonomous systems in real-life applications.
\end{abstract}

{\small 2010 Mathematics Subject Classification: 57Q05, 05E45, 55M30, 68T40.}

{\small Keywords and phrases: Simplicial complex, barycentric subdivision, contiguous simplicial maps, motion planning.}

\section{Introduction}
For a topological space $X$, let $P(X)$ stand for the free path space on $X$ endowed with the compact-open topology. Farber's topological complexity $\TC(X)$ is defined as the sectional category of the evaluation map $e\colon P(X)\to X\times X$, $e(\gamma)=(\gamma(0),\gamma(1))$. Here we use the reduced form of the resulting homotopy invariant, namely a contractible space has zero topological complexity. Thus $\TC(X)+1$ is the smallest cardinality of open covers $\{U_i\}_i$ of $X\times X$ so that $e$ admits a continuous section $\sigma_i$ on each $U_i$. The open sets $U_i$ in such an open cover are called {\it local domains}, the corresponding sections $\sigma_i$ are called {\it local rules}, and the family of pairs $\{(U_i,\sigma_i)\}$ is called a {\it motion planner} for $X$. A motion planner is said to be optimal if it has $\TC_{s}(X)+1$ local domains. In view of the continuity requirement on local rules, an optimal motion planner for the configuration space of a given robot gives us a way to minimize the possibility of accidents in the programing of the robot's performance in noisy environments. In short, topological complexity provides us with a topological framework for studying the motion planning problem in robotics. 

\smallskip
Due to its homotopy nature, Farber's idea quickly attracted the attention of topologists, who began to develop the theoretical aspects of topological complexity. In particular, a number of methods have emerged to estimate $\TC(X)$ for families of spaces~$X$. For instance, homology methods have proven to be useful (and accesible) for bounding from below $\TC(X)$, while more sophisticated (but hard to deal with) homotopy methods have been used to get upper bounds. In some cases, the power of the algebraic topology toolbox leads to the actual computation of $\TC(X)$ without, however, giving a clue about how to construct explicit motion planners. Such a characteristic of the current $\TC$-development has perhaps been the main obstruction for the actual applicability of the $\TC$ ideas to problems arising from real-life needs.

\smallskip
The present paper aims at mending the above situation. Our goal is to lay the theoretical grounds for constructing (potentially optimal) motion planners through computer-implementable algorithms and heuristics. The idea is to use computational topology methods in order to replace the hard (often prohibitive) homotopy considerations for estimating $\TC(X)$ from above. 

\smallskip
A previous attempt to discretize Farber's TC appeared in~\cite[Example~4.5]{tanaka}, where topological complexity is developed in the context of finite spaces. However, the resulting concept appears to be rigid, and in fact it fails to detect the well known equality $\TC(S^1)=1$. Indeed, the best estimate coming from Tanaka's model is $\TC(S^1)\leq3$. In contrast, the use of the barycentric subdivision functor in our model yields a much more flexible invariant. In particular, we are able to recast the equality $\TC(S^1)=1$ in purely combinatorial terms.

\smallskip
Our approach is fully algorithmizable and can be implemented in a computer in order to explore the topological complexity of compact polyhedra.

\smallskip
Our idea rests on the observation that the sectional category of a fibration $p\colon E\to B$ over a CW complex $B$ can be defined in terms of the existence of local sections of $p$ on the elements of a covering of $B$ by Euclidean neighborhood retracts (e.g.~subcomplexes) ---instead of by open sets. In particular, in the case of the fibration defining $\TC(X)$, the following result, whose proof is elementary (compare to~\cite[Lemma 4.21]{F}), allows us to reduce the resulting sectioning problem to a standard homotopy problem, which will later be translated into purely simplicial terms.

\begin{lema}\label{motivation} The evaluation map $e\colon P(X)\to X\times X$ admits a section on a subset $A$ of $X\times X$ if and only if the two compositions $A \hookrightarrow{}X\times X \xrightarrow{\pi_1} X$ and $A \hookrightarrow{}X\times X \xrightarrow{\pi_2} X$ are homotopic.
\end{lema}

\section{Simplicial complexity}
We work with (abstract) simplicial complexes $K$, referred here as ``complexes'', and their topological realizations $\|K\|$. With an eye on applications, we will only care about finite complexes. In addition, the finiteness hypothesis will allow us to prove that the discretized topological complexity recasts the original topological concept.

\smallskip
As suggested by Lemma~\ref{motivation}, we will need to consider a simplicial structure on a product of complexes in such a way that the topological realization of the product recovers the product of the topological realizations of the factors. Consequently, all complexes we deal with will be assumed to be ordered, and their product will be taken in the category of ordered complexes (see for instance~\cite{ES52} for the classical details on the construction). However, we stress that maps of complexes will \emph{not} be required to preserve the given orderings.

\smallskip
The notion of contiguity of simplicial maps (see for instance~\cite[Chapter 3.5]{Spanier}) is recast next in a form suitable for the computational applications we have in mind.

\begin{definition}
For a non-negative integer $c$, two maps of complexes $f,g: K\to L$ are called $c$-contiguous if there is a sequence of maps $h_0, h_1, \cdots, h_c: K\to L$, with $h_0=f$ and $h_c=g$, such that for each $i\in\{1,2,\ldots,c\}$ and for each simplex $s$ of $K$, $h_{i-1}(s)\cup h_i(s)$ is a simplex of $L$ (i.e.~so that each pair $(h_{i-1},h_i)$ is contiguous).
\end{definition}

A collection $\mathcal{C}$ of subcomplexes of $K$ is a cover if $K=\bigcup_{L\in\mathcal{C}}L$ (of course, in such a case, $\bigcup_{L\in\mathcal{C}}\|L\|=\|K\|$). For a non-negative integer $b$, let $\bsd^b$ stand for the $b$-fold iterated barycentric subdivision functor, and fix a map
\begin{equation}\label{laaproxtion}
\iota\colon \bsd^b(K\times K)\to \bsd^{b-1}(K\times K)
\end{equation}
which approximates the identity on $\|K\|\times\|K\|$. By abuse of notation, iterated compositions of these maps will also be denoted by $\iota\colon\bsd^b(K\times K)\to \bsd^{b'}(K\times K)$. Lastly, for $i\in\{1,2\}$, let $\pi_i\colon\bsd^b(K\times K)\to K$ denote the composite of $\iota\colon\bsd^b(K\times K)\to K\times K$ with the $i$-th projection map $K\times K\to K$. By~\cite[Lemmas~3.5.1 and~3.5.4]{Spanier}, the contiguity class of each of these maps is well-defined.

\begin{definition}\label{scbc}
For non-negative integers $b$ and $c$, the $(b,c)$-simplicial complexity of an (ordered) complex $K$, $\SC^b_c(K)$, is one less than the smallest cardinality of finite covers of $\bsd^b(K\times K)$ by subcomplexes $J$ for which the compositions
$$\xymatrix{
J\; \ar@{^{(}->}[r] & \bsd^b(K\times K) \ar[r]^<<<<{\pi_1} & K & \mathrm{and} & J\; \ar@{^{(}->}[r] & \bsd^b(K\times K) \ar[r]^<<<<{\pi_2} & K
}$$
are $c$-contiguous. When no such finite coverings exist, we set $\SC^b_c(K)=\infty$.
\end{definition}

\begin{remark}{\em
The ordering in $K$ is used only for the construction of the suitable simplicial structure on $K\times K$; the value of $\SC^b_c(K)$ is clearly independent of the chosen ordering.
}\end{remark}

Since the topological realizations of $c$-contiguous maps are homotopic, Lemma~\ref{motivation} and~\cite[Proposition~4.9]{F} yield
\begin{equation}\label{hasus}
\SC^b_c(K)\geq\TC(\|K\|).
\end{equation}

The obvious monotonic sequence $\SC_0^b(K)\geq\SC_1^b(K)\geq\SC_2^b(K)\geq\cdots\geq0$ is eventually constant, and we let $\SC^b(K)$ stand for the corresponding stable value. Note that the sequence of numbers $\SC^b_c(K)$ depends, in principle, on the chosen approximations $\iota$. However, we prove:

\begin{lema}\label{sceindaprox}
The stabilized value $\SC^b(K)$ is independent of the chosen approximations~(\ref{laaproxtion}). 
\end{lema}
\begin{proof}
Let $\overline{\SC}^b_c$ stand for the invariant defined in terms of another set of approximations $\overline{\iota}\colon \bsd^b(K\times K)\to \bsd^{b-1}(K\times K)$ to the identity. By~\cite[Lemmas~3.5.1 and~3.5.4]{Spanier}, the corresponding compositions $\pi_i,\overline{\pi_i}\colon\bsd^b(K\times K)\to K$, as well as their restrictions $\pi_i\circ j$ and $\overline{\pi_i}\circ j$, are 1-contiguous, where $j\colon J\hookrightarrow\bsd^b(K\times K)$ is the inclusion, and $i\in\{1,2\}$. Consequently $\SC^b_c(K)\geq\overline{\SC}^b_{c+2}(K)\geq\SC^b_{c+4}(K)$, and the result follows.
\end{proof}

Following the idea in the previous proof, note that in the setting of Definition~\ref{scbc}, if $\lambda_J\colon\bsd(J)\to J$ is a simplicial approximation to the identity, then the two compositions in the diagram 
$$\xymatrix{
J\;\ar@{^{(}->}[r] & \bsd^b(K\times K) \\
\bsd(J)\;\ar@{^{(}->}[r] \ar[u]^\lambda& \bsd^{b+1}(K\times K)\ar[u]_\iota
}$$
are 1-contiguous, as they are simplicial approximations to the inclusion $\|J\|\subseteq\|K\|\times\|K\|$. Consequently $\SC^b_c(K)\geq\SC^{b+1}_{c+2}(K)$ and
\begin{equation}\label{stabb}
\SC^0(K)\geq\SC^1(K)\geq\SC^2(K)\geq\cdots.
\end{equation}

\begin{definition} The simplicial complexity of a complex $K$, $\SC(K)$, is the stabilized value of the monotonic sequence~(\ref{stabb}).
\end{definition}

We stress the fact that the equality $\SC(K)=\SC^b_c(K)$ holds for large enough indices $b$ and $c$ (depending on $K$), so that~(\ref{hasus}) becomes
\begin{equation}\label{bascom}
\SC(K)\geq\TC(\|K\|).
\end{equation}
In the light of~\cite[Theorems~3.5.6 and~3.5.8]{Spanier}, it is reasonable to expect that the latter inequality is sharp. This fact is illustrated in Section~\ref{toy} for the circle, and is proved next in detail for any $K$.

\begin{theo}\label{sc=tc}
Equality holds in~(\ref{bascom}) for any finite $K$.
\end{theo}

\begin{proof}
Let $\TC(\|K\|)=k$ and choose a motion planner $\{(U_0,\sigma_0),(U_1,\sigma_1),\ldots,(U_k,\sigma_k)\}$ for $\|K\|$ with $k+1$ local domains. Choose a large positive integer $b$ so that the realization of each simplex of $\bsd^b(K\times K)$ is contained in some $U_j$ ($0\leq j\leq k$) ---this uses the finiteness assumption on $K$. For each $j\in\{0,1,\ldots,k\}$, let $L_j$ be the subcomplex of $\bsd^b(K\times K)$ consisting of those simplices whose realization is contained in $U_j$. Then $L_0,L_1,\ldots,L_k$ cover $K\times K$. 

By Lemma~\ref{motivation} the two projections $\pi_1,\pi_2\colon\|K\|\times\|K\|\to\|K\|$ are homotopic over each $U_i$ and, in particular, they are homotopic over the (realization of the) corresponding subcomplex $L_i$. By~\cite[Lemma~3.5.4 and Theorem~3.5.6]{Spanier}, there are positive integers $b'$ and $c$ such that, for each $j\in\{0,1,\ldots,k\}$, the two simplicial composites 
$$\xymatrix{
\bsd^{b+b'}(L_i)\, \ar@{^{(}->}[r] & \bsd^{b+b'}(K\times K) \ar@<.5ex>[r]^>>>>>{\pi_1\,}
\ar@<-.5ex>[r]_>>>>>{\pi_2\,} & \,K
}$$
are $c$-contiguous. It follows that $\SC^{b+b'}_c(K)\leq k$, which implies equality in~(\ref{bascom}).
\end{proof}

Note that Theorem~\ref{sc=tc} asserts that, when the configuration space of a robot has a simplicial structure, it is always possible to produce optimal motion planners whose local rules are piecewise linear. The search of such motion planners can then be done by computer. Exhaustive search, however, is most likely domed to fail, as the size of the search space increases exponentially with every subdivision. Heuristic-based algorithms are thus eagerly awaited.

\smallskip
Theorem~\ref{sc=tc} allows us to extrapolate all the nice properties of Farber's topological complexity to the simplicial realm. For instance:
\begin{enumerate}[(i)]
\item Two complexes whose topological realizations are homotopy equivalent necessarily have the same simplicial complexity. 
\item A complex $K$ has $\SC(K)=0$ ($\SC(K)=1$) if and only if the realization $\|K\|$ is contractible (has the homotopy type of an odd sphere).
\item The simplicial complexity of an (ordered) product of (ordered) complexes $K_i$ is bounded from above by $\sum_i\SC(K_i)$.
\end{enumerate}
Since we have restricted ourselves to finite complexes, the homotopy conditions in (i) and (ii) above regarding the topological realizations of complexes can equally be stated in terms of the notion of strong equivalence of complexes.

\smallskip
We close this section by remarking that the approach in this paper is similar to that used in~\cite{MR3404603} for discretizing the Lusternik-Schnirelman category of a space $X$.

\section{An example: the circle}\label{toy}
We think of the circle $S^1$ as the realization of the complex $S_1$ given as the 1-dimensional skeleton of the 2-dimensional simplex $\Delta^2$. Let the vertices of $S_1$ be $0$, $1$, and $2$ ordered in the obvious way. The (realization of the ordered) product structure on $S_1\times S_1$ can be depicted as
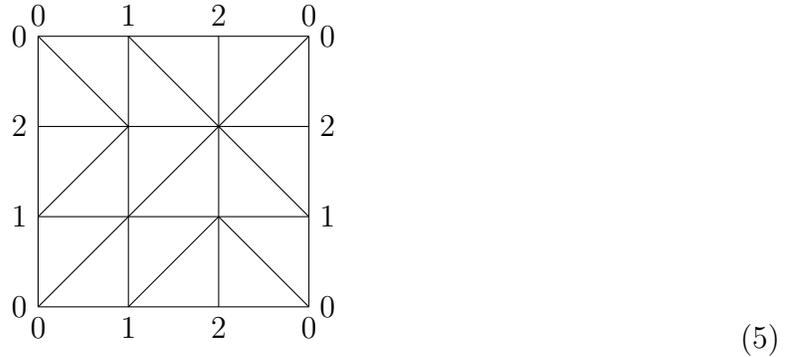
\begin{equation}\label{s1xs1}
\begin{tikzpicture}[x=.6cm,y=.6cm]
\draw(0,0)--(6,0); \draw(0,0)--(0,6);
\draw(0,6)--(6,6); \draw(6,0)--(6,6);
\draw(0,0)--(6,6); \draw(6,2)--(2,6); 
\draw(0,2)--(6,2); \draw(0,4)--(6,4); 
\draw(2,0)--(2,6); \draw(4,0)--(4,6);
\draw(2,0)--(4,2); \draw(4,2)--(6,0);
\draw(0,2)--(2,4); \draw(2,4)--(0,6);
\node [below] at (0,0) {0};
\node [below] at (2,0) {1};
\node [below] at (4,0) {2};
\node [below] at (6,0) {0};
\node [left] at (0,0) {0};
\node [left] at (0,2) {1};
\node [left] at (0,4) {2};
\node [left] at (0,6) {0};
\node [above] at (0,6) {0};
\node [above] at (2,6) {1};
\node [above] at (4,6) {2};
\node [above] at (6,6) {0};
\node [right] at (6,0) {0};
\node [right] at (6,2) {1};
\node [right] at (6,4) {2};
\node [right] at (6,6) {0};
\end{tikzpicture}
\end{equation}
where opposite sides of the external square are identified as indicated. The corresponding barycentric subdivision starts as
$$
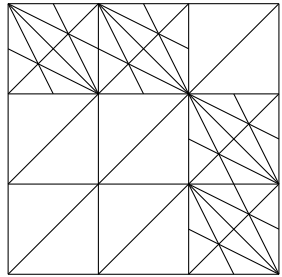
\begin{tikzpicture}[x=.6cm,y=.6cm]
\draw(0,0)--(6,0); \draw(0,0)--(0,6);
\draw(0,6)--(6,6); \draw(6,0)--(6,6);
\draw(0,0)--(6,6); \draw(6,2)--(2,6); 
\draw(0,2)--(6,2); \draw(0,4)--(6,4); 
\draw(2,0)--(2,6); \draw(4,0)--(4,6);
\draw(2,0)--(4,2); \draw(4,2)--(6,0);
\draw(0,2)--(2,4); \draw(2,4)--(0,6);
\draw(4,2)--(6,4); \draw(6,2)--(4,3);
\draw(4,4)--(6,0); \draw(4,4)--(6,3);
\draw(6,2)--(5,4); \draw(4,0)--(6,2);
\draw(4,2)--(6,1); \draw(4,2)--(5,0);
\draw(4,1)--(6,0); \draw(1,4)--(0,6);
\draw(2,4)--(4,6); \draw(2,6)--(3,4);
\draw(4,4)--(0,6); \draw(4,4)--(3,6);
\draw(2,6)--(4,5); \draw(0,4)--(2,6);
\draw(2,4)--(1,6); \draw(2,4)--(0,5);
\end{tikzpicture}$$
where we have only shown the subdivision of the four ``squares'' in~(\ref{s1xs1}) whose diagonal has a negative slope ---but all other squares are to be subdivided in a similar way. Note that the subcomplex $K_0$ (whose realization is) shaded in gray in
$$
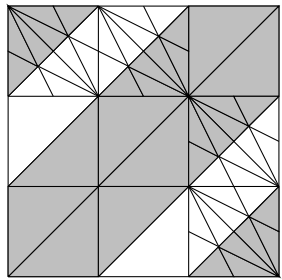
\begin{tikzpicture}[x=.6cm,y=.6cm]
\draw(0,0)--(6,0); \draw(0,0)--(0,6); \draw(0,6)--(6,6); \draw(6,0)--(6,6); \draw(0,0)--(6,6); \draw(6,2)--(2,6); \draw(0,2)--(6,2); \draw(0,4)--(6,4); \draw(2,0)--(2,6); \draw(4,0)--(4,6); \draw(2,0)--(4,2); \draw(4,2)--(6,0); \draw(0,2)--(2,4); \draw(2,4)--(0,6);\draw(4,2)--(6,4); \draw(6,2)--(4,3); \draw(4,4)--(6,0); \draw(4,4)--(6,3); \draw(6,2)--(5,4); \draw(4,0)--(6,2); \draw(4,2)--(6,1); \draw(4,2)--(5,0); \draw(4,1)--(6,0); \draw(1,4)--(0,6); \draw(2,4)--(4,6); \draw(2,6)--(3,4); \draw(4,4)--(0,6); \draw(4,4)--(3,6);
\draw(2,6)--(4,5); \draw(0,4)--(2,6); \draw(2,4)--(1,6); \draw(2,4)--(0,5); \draw[fill=lightgray](0,0)--(2,0)--(2,2)--(0,0); \draw[fill=lightgray](0,0)--(0,2)--(2,2)--(0,0); \draw[fill=lightgray](0,2)--(2,2)--(2,4)--(0,2); \draw[fill=lightgray](2,2)--(2,4)--(4,4)--(2,2); \draw[fill=lightgray](2,2)--(4,2)--(4,4)--(2,2); \draw[fill=lightgray](2,0)--(2,2)--(4,2)--(2,0); \draw[fill=lightgray](4,4)--(4,6)--(6,6)--(4,4); \draw[fill=lightgray](4,4)--(6,4)--(6,6)--(4,4); \draw[fill=lightgray](2,4)--(3,4)--(2.66,4.66)--(2,4);
\draw[fill=lightgray](2.66,4.66)--(3,4)--(4,4)--(2.66,4.66); \draw[fill=lightgray](2.66,4.66)--(4,4)--(3,5)--(2.66,4.66); \draw[fill=lightgray](4,4)--(3,5)--(3.34,5.33)--(4,4); \draw[fill=lightgray](4,4)--(3.34,5.33)--(4,5)--(4,4); \draw[fill=lightgray](4,6)--(3.34,5.33)--(4,5)--(4,6); \draw[fill=lightgray](4,4)--(5,3)--(5.33,3.34)--(4,4); \draw[fill=lightgray](4,4)--(5.33,3.34)--(5,4)--(4,4); \draw[fill=lightgray](6,4)--(5.33,3.34)--(5,4)--(6,4); \draw[fill=lightgray](4,2)--(4,3)--(4.66,2.66)--(4,2); \draw[fill=lightgray](4.66,2.66)--(4,3)--(4,4)--(4.66,2.66); \draw[fill=lightgray](4.66,2.66)--(4,4)--(5,3)--(4.66,2.66);
\draw[fill=lightgray](4-4,4+2)--(5-4,3+2)--(5.33-4,3.34+2)--(4-4,4+2);
\draw[fill=lightgray](4-4,4+2)--(5.33-4,3.34+2)--(5-4,4+2)--(4-4,4+2);
\draw[fill=lightgray](6-4,4+2)--(5.33-4,3.34+2)--(5-4,4+2)--(6-4,4+2);
\draw[fill=lightgray](4-4,2+2)--(4-4,3+2)--(4.66-4,2.66+2)--(4-4,2+2);
\draw[fill=lightgray](4.66-4,2.66+2)--(4-4,3+2)--(4-4,4+2)--(4.66-4,2.66+2); \draw[fill=lightgray](4.66-4,2.66+2)--(4-4,4+2)--(5-4,3+2)--(4.66-4,2.66+2); \draw[fill=lightgray](4+2,4-4)--(3+2,5-4)--(3.34+2,5.33-4)--(4+2,4-4); \draw[fill=lightgray](4+2,4-4)--(3.34+2,5.33-4)--(4+2,5-4)--(4+2,4-4); \draw[fill=lightgray](4+2,6-4)--(3.34+2,5.33-4)--(4+2,5-4)--(4+2,6-4); \draw[fill=lightgray](2+2,4-4)--(3+2,4-4)--(2.66+2,4.66-4)--(2+2,4-4); \draw[fill=lightgray](2.66+2,4.66-4)--(3+2,4-4)--(4+2,4-4)--(2.66+2,4.66-4); \draw[fill=lightgray](2.66+2,4.66-4)--(4+2,4-4)--(3+2,5-4)--(2.66+2,4.66-4);\end{tikzpicture}$$
collapses to the diagonal $\Delta$ in $S_1\times S_1$, whereas the realization of the complex $K_1$ left unshaded deformation retracts to the antidiagonal $\Delta'=\{(x,-x)\}$ in $S^1\times S^1$. Since $S^1$ admits (local) motion planners both in $\|\Delta\|$ and $\Delta'$, Lemma~\ref{motivation} (together with the properties relating homotopy maps with the contiguity classes of their simplicial approximations) yields $\SC(S_1)\leq1$. Equality then follows from~(\ref{bascom}). Note that in this case~Theorem~\ref{sc=tc} holds in the stronger form $\SC^1(S_1)=\SC^2(S_1)=\cdots=\TC(S^1)=1$.

%\bibliographystyle{plain}
%\bibliography{bib}

\end{document}